\documentclass[a4paper]{amsart}

\title{Trefoils and hexafoils in 3-braids}
\author{Sebastian Baader, Levi Ryffel}
\date{\today}

\usepackage{amsmath}
\usepackage{amsfonts}
\usepackage{amssymb}
\usepackage{amsthm}

\usepackage{enumerate}
\usepackage{graphicx}
\usepackage[hidelinks]{hyperref}

\newtheorem{theorem}{Theorem}

\newtheorem{lemma}{Lemma}

\theoremstyle{definition}

\begin{document}
\maketitle

\begin{abstract}
We determine the cobordism distance between all positive 3-braid links and connected sums of trefoil knots, up to a constant error. This may be seen as a reinterpretation of Coxeter's finiteness result on the braid group on three strands modulo third powers of the generators. We go on to explore the limits of Coxeter's result, and find the cobordism distance between three strand torus links and almost all connected sums of two strand torus links, up to a constant error.
\end{abstract}

\section{Introduction}

In the late 50's, Coxeter proved that the quotient of the braid group $B_3$ by the $k$-th power of the generators is finite, provided $k \leq 5$, and infinite otherwise~\cite{coxeter}. An important step in the proof is to reduce large powers of the central element $(\sigma_1 \sigma_2)^3 \in B_3$ to the trivial braid, by successively removing instances of $\sigma_1^k$ or $\sigma_2^k$. As we will see, the same technique allows us to find virtually minimal genus cobordisms between all positive torus links of type $T(3,m)$ and connected sums of torus links of type $T(2,k)$, for $k \leq 5$. Surprisingly, the argument extends to the case $k=6$, despite the fact that the centre of $B_3$ modulo sixth powers of the generators is infinite. In order to state our results, we need the cobordism distance $d_\chi(K,L)$ of oriented links $K,L \subset S^3$, defined as the absolute value of the maximal Euler characteristic among all smooth cobordisms in $S^3 \times [0,1]$ without closed components between $K \subset S^3 \times\{0\}$ and $L \subset S^3 \times\{1\}$~\cite{baader12}. The links we consider are closures of positive 3-braids. We use the notation $B_3^+$ for the monoid of positive 3-braids, and $\widehat{\beta} \subset S^3$ for the standard closure of positive 3-braids $\beta \in B_3^+$.

\begin{theorem}\label{thm:3-braids}
  There exists a constant $C_3 \leq 18$ such that
  for all non-split positive braids $\beta \in B_3^+$ of length $\ell(\beta)$ and all $n \in \mathbb{N}$ we have
  \[
  d_\chi(\widehat \beta, T(2, 3)^n)=\frac{\ell(\beta)}{3}+2\left|n-\frac{\ell(\beta)}{3}\right|+E_3(\beta, n),
  \]
  with an error term $E_3(\beta, n)$ satisfying $|E_3(\beta, n)| \leq C_3$.
\end{theorem}

For our second result, we need to specify the meaning of the iterated connected sum of torus links $T(2,2k)^n$; our convention is that all summands are attached to one distinguished link component, as in the well-defined connected sum of knots $T(2,3)^n$. In particular, our version of the connected sum $T(2,2k)^n$ is not the same as the closure of the braid $\sigma_1^{2k} \sigma_2^{2k} \cdots \sigma_n^{2k}$.

\begin{theorem}\label{thm:3-torus-knots}
  There exist constants
  $C_4, C_5, C_6 < 44$ such that
  for all $m \geq 1$, all $n \geq 0$,
  and all $k \in \{4, 5, 6\}$
  we have
  \[
    d_\chi(T(3, m), T(2, k)^n) = \frac{2m}{k} + (k-1) \left|n - \frac{2m}{k} \right|
    + E_k(m, n)
  \]
  with an error term $E_k(m, n)$ satisfying
  $|E_k(m, n)| \leq C_k$.
\end{theorem}

As indicated above, the case $k = 6$ about sums of the hexafoil link $T(2,6)$ can be interpreted as an `affine limit case' in Coxeter's factor groups. The cobordism distance between $T(3,m)$ and $T(2,k)^n$ for larger parameters, i.e. for $k \geq 7$, does not admit an obvious analogous formula. Nevertheless, for all fixed $m$ and $k \geq 7$, we are able to determine the said cobordism distance for almost all values of~$n$.

\begin{theorem}\label{thm:large-k}
  There exists a constant $C < 44$ such that for all $k \geq 7$
  and all $n, m \geq 0$ there exists an error term $E_k(m, n)$
  satisfying $|E_k(m, n)| \leq C$ we have:
  \begin{enumerate}[\normalfont(i)]
    \item If $n \geq m/3$, then
      \[
        d_\chi(T(3, m), T(2, k)^n) = (k-1)n - 4m/3 + E_k(m, n).
      \]
    \item If $n \leq 5m/3k - (k+4)$, then
      \[
        d_\chi(T(3, m), T(2, k)^n) = 2m - (k-1)n + E_k(m, n).
      \]
  \end{enumerate}

\end{theorem}

The extreme case of large $k$ can be understood via Feller's result on the cobordism distance between torus links of type $T(3,m)$ and $T(2,k)$~\cite{feller}. Apart from that, our results are independent of Feller's result.

The proof of the three above theorems involves a construction of smooth cobordisms of suitable genus, and lower bounds certifying that these cobordisms have essentially minimal genus.
The constructive part is inspired by
Truöl's recent results on the upsilon invariant of $3$-braid knots~\cite{truoel}.
As so often in the context of cobordisms, the lower bounds are given by evaluations of the Levine-Tristram signature functions $\sigma_\omega(L)$, defined for all links $L \subset S^3$ and (almost) all $\omega \in S^1$~\cite{levine, tristram}. The relevant values of $\omega$ to bound the cobordism distance between $T(3,m)$ and $T(2,k)^n$ for $k=3,4,5$ are in the vicinity of
\[
\omega=\zeta_6,\zeta_4,\zeta_{10}^3,
\]
respectively, where $\zeta_N$ denotes the $N$-th primitive root of unity.
For larger values of $k$, the relevant value of $\omega$
is $-1$, which corresponds to the classical signature invariant.
The two main steps of the proofs, constructing minimal cobordisms, and finding lower bounds, are contained in the following two sections, in this order.
The actual proofs are presented in the last section.

\section{Constructing key cobordisms}\label{sec:finding-subsurfaces}
The key technique
in constructing minimal cobordisms
is to find many subsurfaces
of the right type
in the Seifert surface of
the braid in question.
More precisely, we transform
a braid $\beta$ to the trivial braid
by deleting powers $a^k$ or $b^k$
(corresponding to subsurfaces of
type $T(2, k)$)
of generators, for fixed $k \geq 3$,
while at the cost of an error we are
allowed to also delete single generators.
We summarize this strategy in the following result.

\begin{lemma}\label{lem:delete-power}
  Let $k \geq 1$,
  let $\beta \in B_3^+$ be a positive braid, and let
  $\beta'$ be obtained from $\beta$ by removing
  a $k$-th power $a^k$ or $b^k$ of a generator
  $n$ times,
  and by removing an arbitrary number of single generators.
  Then
  \[
    d_\chi(\widehat \beta, T(2, k)^n) \leq \ell(\beta) - (k-1)n + 2.
  \]
\end{lemma}

\begin{proof}
  We first construct a cobordism of Euler characteristic $-1$
  between $\widehat \beta$ and  $\widehat \beta' \# T(2, k)$,
  where $\beta'$ is obtained from $\beta$ by deleting
  a $k$-th power of a generator.
  Note that the notation $\widehat \beta' \# T(2, k)$
  is ambiguous as soon as $\widehat \beta'$ has more than one component.
  This ambiguity will be resolved at the end of the proof.
  As indicated in Figure~\ref{fig:delete-power},
  a cobordism as described above can be obtained by cutting
  the canonical Seifert surface of $\beta$ along an arc.
  Because $T(2, 1)$ is the trivial knot,
  the special case $k = 1$ shows in particular
  that the deletion of a generator
  yields a cobordism of Euler characteristic $-1$.

  Because the summand $T(2, k)$ can be moved
  freely
  along the strand it is connected to,
  this allows iteration of this removal,
  so that
  if $\beta'$ is obtained from $\beta$ by deleting
  $n$ instances of $k$-th powers of $a$ or $b$
  and $j$ instances of single generators,
  then there exists a cobordism
  of Euler characteristic $-(n+j)$
  between $\widehat \beta$
  and
  a connected sum of $\beta'$ with
  $n$ many copies of $T(2, k)$.
  Deleting the generators in $\beta'$ one by one
  yields a split union of three powers of $T(2, k)$,
  which can be merged
  to the sum $T(2, k)^n$
  using two saddle moves yielding a cobordism of Euler characteristic $-2$.
  We thus obtain a cobordism
  of Euler characteristic
  \[ \chi = -(n + j + \ell(\beta') + 2) = -(\ell(\beta) - (k-1)n + 2)\]
  between $\widehat \beta$ and $T(2, k)^n$.
\end{proof}

\begin{figure}[htb]
  \centering
  \begin{minipage}{0.33\textwidth}
    \centering
    \includegraphics{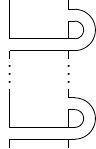}
  \end{minipage}%
  \begin{minipage}{0.33\textwidth}
    \centering
    \includegraphics{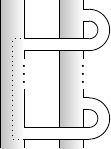}
  \end{minipage}%
  \begin{minipage}{0.33\textwidth}
    \centering
    \includegraphics{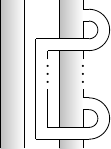}
  \end{minipage}%
  \caption{The braid $a^k$ or $b^k$,
  its canonical Seifert surface, and
the surface obtained from cutting along
the dotted arc}%
  \label{fig:delete-power}
\end{figure}

\subsection{Third powers}
Our goal in this subsection is to
start with a positive braid $\beta \in B_3^+$
and, by removing third powers of generators
and a constant number of single generators,
transform it into the trivial braid.
We first iterate the procedure of writing
$\beta$ in Garside normal form,
and then removing all emerging third powers
of generators.
This process terminates once
the Garside normal form
has no third powers
of generators.
After conjugation
and removing
at most $3$ single generators,
the Garside normal form~\cite{garside} of
the resulting braid
can be assumed to be of the form
\(
  \Delta^{i} (a^{2} b^{2})^{j},
\)
where $\Delta = aba$.
Indeed, Proposition~3.2 in~\cite{truoel}
directly implies that,
up to conjugation, any
$3$-braid can be written
as a product $\Delta^i a^{p_1} b^{q_1}
\cdots a^{p_r} b^{q_r} \beta'$,
where $p_j, q_j \geq 2$ for all $j$,
and
where $\beta'$ is either a power of $a$
or equal to $a^p b$ for some $p \in \{1, 2, 3\}$.

Consider the braid
$(a^2b^2)^\ell$.
Upon deletion of a single generator,
it can be transformed into
\[
  (a^2 b^2)^{\ell_1} a b^2 (a^2 b^2)^{\ell_2}
  = (a^2 b^2)^{\ell_1 - 1} a \Delta^2 (a^2b^2)^{\ell_2}
  = \Delta^2 (a^2 b^2)^{\ell_1 - 1} a (a^2 b^2)^{\ell_2},
\]
for any $\ell_1, \ell_2$ with
$\ell_1 + \ell_2 = \ell - 1$.
Removing a third power of $a$ and
then a third power of $b$ similarly leads to
\[
  \Delta^2(a^2 b^2)^{\ell_1-2} a^2 b (a^2 b^2)^{\ell_2 - 1}
  = \Delta^4 (a^2 b^2)^{\ell_1 - 2} b (a^2 b^2)^{\ell_2 - 2},
\]
from which we may again remove two third powers
of generators to obtain
the braid $\Delta^4 (a^2 b^2)^{\ell_1-3}a b^2 (a^2 b^2)^{\ell_2 - 3}$.
This process can be iterated until we are left
with a positive braid of the form
$\Delta^{2i} \beta'$
for some $i$,
where $\beta'$ is a positive braid
of length at most $5$ that contains a half
twist $\Delta$ as soon as its length is
at least $3$.
We conclude that after deleting at most
$3$ single generators, the braid
$(a^2 b^2)^\ell$ can be transformed into
a braid of the form $\Delta^i$
for some $i$.


Next, consider the braid $\Delta^i$.
As indicated
by the the fact that
\[
  (ab)^6 =
  (aba) b (aba) b (aba) b
  =
  b a b^3 a b^3 a b^2,
\]
it is possible
to delete third powers $a^3$ and $b^3$
of generators from $\Delta^4 = (ab)^6$ to obtain
the trivial braid,
The fact that
\[
  (ab)^4 =
  (aba) b (aba) b =
  b a b^3 a b^2
\]
shows that two third powers
of a generator
can be removed from $\Delta^3$,
showing that upon deletion of at most
$4$ single generators (the worst case being $(ab)^5$),
the braid
$\Delta^i$ can be transformed to the
trivial braid by deleting third powers
of generators.


Note that in order to reduce $\beta$
to the trivial braid by removing third powers of generators,
we had to delete at most $3 + 3 + 4 = 10$ generators.
By Lemma~\ref{lem:delete-power},
we have thus shown that
for any positive braid $\beta \in B_3$,
and any $n \geq 0$
with $3n \leq \ell(\beta) - 10$ we have
\[
  d_\chi(\widehat \beta, T(2, 3)^n)
  \leq \ell(\beta) - 2n + 2.
\]

\subsection{Fourth powers}
As shown in~\cite{bbl}, we have
\[
  \Delta^8 = (ab)^{12} = a b^3 a b^5 a b^3 a b^4 a b^4.
\]
After removing three fourth powers of $b$,
the braid is transformed into
\[
  a b^3 (a b a) b^3 a^2 = a b^4 a b^4 a^2,
\]
which is transformed into the trivial braid
after further removal of three fourth powers of generators.

It is possible to remove $2$ single generators from $(ab)^3$
to obtain $b^4$.
Thus, one can remove $4$ single generators from $(ab)^4$
to obtain $b^4$.

Similarly,
note that
\[
  (ab)^5 =
  a (bab) (aba) (bab)
  =
  a^2 (bab) (aba) ba
  =
  a^3 (bab) a b^2 a
  =
  a^4 b a^2 b^2 a,
\]
which, after removing a fourth power of $a$, is equal to
\(
  b a^2 b^2 a.
\)
This braid is elementarily conjugate to
\[
  ab a^2 b^2 = bab a b^2 = b^2 a b^3,
\]
which again is elementarily conjugate to $a b^5$.
It is thus possible to remove
$2$ single generators from $(ab)^5$,
$4$ single generators from $(ab)^6$,
or $6$ generators from $(ab)^7$ to
obtain a braid that can be transformed into the trivial braid
by removing two fourth powers of generators.

Next, similar techniques show
that three fourth powers of generators can be removed from
$(ab)^8$.
It is thus enough to remove $4$ single generators from $(ab)^8$
or $6$ generators from $(ab)^9$.
Finally, removing $4$ single generators from $(ab)^{10}$ or
$6$ single generators from $(ab)^{11}$ is enough.

Thus, the braid $(ab)^m$ can be transformed into the trivial braid
by deleting at most $6$ single generators, the worst case being
$(ab)^j$ for $j \in \{7, 9, 11\}$.
Lemma~\ref{lem:delete-power}
implies that for any $n \geq 0$ with $4n \leq 2m - 6$ that
\[
  d_{\chi}(T(3, m), T(2, 4)^n) \leq 2m - 3n + 2.
\]

\subsection{Fifth powers}
Again as shown in~\cite{bbl},
\[
  \Delta^{4i} = (ab)^{6i} = ab^3 (ab^5)^{i-1} a b^3 ab^4 (ab^5)^{i-2} ab^4
\]
for all $i \geq 2$.
Removing fifth powers from this braids results in
\(
  ab^3 a^{i}b^3 a b^4 a^{i-1} b^4,
\)
which after deleting at most $7$ generators and further removal of
fifth powers is equal to
\(
  ab a b^3 = bab^4,
\)
which upon deletion of a further single generator is equal to
the fifth power $b^5$.

Similarly,
\[
  \Delta^{4i + 2} = (ab)^{6i + 3} = a b^3 (a b^{5})^{i-1} a b^4 a b^3 (ab^5)^{i-1} a b^4
\]
for all $i \geq 1$,
which after deletion of at most $8$ generators can be transformed,
using deletion of fifth powers of generators, into
$ab^2a b^2$, a braid of length $6$.

Finally, an arbitrary braid $(ab)^m$ with $m \geq 6$
can be transformed either to $\Delta^{4i}$ or $\Delta^{4i+2}$ for some $i$
using at most $5$ deletions of single generators
for a total error of at most $5 + 6 + 8 = 19$.
The braid $(ab)^j$ for $j \leq 5$ is shorter than this maximal error,
so we have shown that for all $n \geq 0$ with $5n \leq 2m - 19$
we have
\[
  d_{\chi}(T(3, m), T(2, 5)^n) \leq 2m - 4n + 2.
\]

\subsection{Sixth powers}
The formula $\Delta^{4i} = ab^3 (ab^5)^{i-1} ab^3 ab^4 (ab^5)^{i-2}ab^4$
from the previous subsection
shows that by deleting $12$ single generators
and removing one sixth power of a generator,
the braid $(ab)^{6i}$ can be transformed into
$(ab^5)^{2i-3}$.
Removing a further sixth power
yields $(ab^{-1})(ab^5)^{2i - 3}$,
and removing two further single generators
transforms $(ab)^{6i}$ into $(ab^{-1})^{2i-4}$
using $14$ deletions of single generators
and $2i-2$ deletions of sixth powers.

Similarly, the formula
$\Delta^{4i+2} = ab^3 (ab^5)^{i-1} ab^4 ab^3 (ab^5)^{i-1} ab^4$
shows that by deleting $12$ single generators
and removing $2i-1$ sixth powers of a generator,
the braid
$(ab)^{6i+3}$ can be transformed into
$(ab^{-1})^{2i-2}$.

The closures of the braids $(ab^{-1})^{2j}$ have cobordism
distance at most two to the unknot.
This is because two cobordisms of Euler characteristic $-1$
transform the closure of $(ab^{-1})^{2j}$
into $L\#L$, where $L$ is the closure of
$(ab^{-1})^j$.
Because $L$ is isotopic to its mirror image,
it follows that $L\#L$ is slice.
Using this as well as Lemma~\ref{lem:delete-power}
and the fact that four deletions of single generators
are enough to turn any power of $ab$ into $\Delta^{4i}$
or $\Delta^{4i+2}$ for some $i$,
we obtain that for all $n \geq 0$ with $6n \leq 2m - 20$
we have
\[
  d_\chi(T(3, m), T(2, 6)^n) \leq 2m - 5n + 2.
\]

\subsection{Larger powers}
From here on we are not quite as efficient as in the previous cases,
in the sense that we cannot, up to bounded error,
turn an arbitrary power of $ab$
into the trivial braid by removing $k$-th powers of generators
for $k \geq 7$.
We nonetheless find a substantial number of $k$-th powers of generators,
although we make no claim that this number is optimal.

Note that the braid
$(ab)^{6i} = ab^3 (ab^5)^{i-1} ab^3 ab^4 (ab^5)^{i-2}ab^4$
can be transformed into $b^{10i-1}$ by deleting all
$2i+1$ occurrences of $a$.
This finds $[(10i-1)/k]$ many $k$-th powers of generators in
$(ab)^{6i}$.

Similarly, the braid
$(ab)^{6i+3} = ab^3 (ab^5)^{i-1} ab^4 ab^3 (ab^5)^{i-1} ab^4$
can be transformed into $b^{10i+4}$ by deleting all
$2i+2$ occurrences of $a$.
This finds $[(10i+4)/k]$ many $k$-th powers of generators in
$(ab)^{6i+3}$.

Finally, any power $(ab)^m$ can be turned into one of the
braids under consideration by deleting at most four single generators.
Using Lemma~\ref{lem:delete-power},
we are now able to conclude that for all $n \geq 0$ with
$n \leq 5m/3k - (k+4)$ we have
\[
  d_{\chi}(T(3, m), T(2, k)^n) \leq 2m - (k-1)n + 2.
\]

\section{Signatures}\label{sec:signatures}
In this section, we provide lower bounds
on the cobordism distance between
$T(3, m)$ (or a positive braid $\beta$) and $T(2, k)^n$
for small $m$ compared to $n$.
To this end, we recall a result proven by
Gambaudo-Ghys, namely Corollary 4.4 in~\cite{gambaudo-ghys}.

\begin{lemma}\label{lem:gambaudo-ghys}
  Let $\beta \in B_3^+$ be a positive braid,
  and
  let $\omega = e^{2 \pi i \theta}$ for a rational number
  $\theta$ with $0 < \theta < 1/3$.
  Then the Levine-Tristram signature at $\omega$ satisfies
  \[
    |\sigma_\omega(\widehat \beta) - 2 \theta \ell(\beta)| \leq 2.
  \]
\end{lemma}

The formulas of Gambaudo-Ghys, specifically Proposition~5.1
in~\cite{gambaudo-ghys}, can be used to explicitly determine
the Levine-Tristram signatures of the links $T(2, k)$
for $k \in \{3, 4, 5\}$, see Figure~\ref{fig:signatures}
for the result.
For a link $L$, we introduce the shorthand notations
\begin{enumerate}[(i)]
  \item $\sigma_3(L) = \lim_{\theta \to 1/6_+} \sigma_{\exp(2 \pi i \theta)}(L)$,
  \item $\sigma_4(L) = \lim_{\theta \to 1/4_+} \sigma_{\exp(2 \pi i \theta)}(L)$,
  \item $\sigma_5(L) = \lim_{\theta \to 3/10_+} \sigma_{\exp(2 \pi i \theta)}(L)$,
  \item $\sigma_k(L) = \sigma_{-1}(L)$ for $k \geq 6$,
\end{enumerate}
in order to facilitate comparison of the signature functions
in specific neighbourhoods.
Note that the limit in $\sigma_k$
for $k \leq 5$
is taken down to the $\theta$ corresponding
to the last jump in the signature function
of $T(2, k)$,
as this is where
we expect to find
the largest difference between $\sigma_\omega(T(3, m))$
and $\sigma_\omega(T(2, k)^n)$ for small $m$.

\begin{figure}[htb]
  \centering
  \begin{minipage}{0.33\textwidth}
    \centering
    \includegraphics{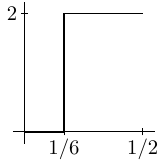}
  \end{minipage}%
  \begin{minipage}{0.33\textwidth}
    \centering
    \includegraphics{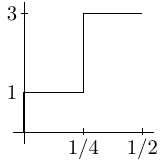}
  \end{minipage}%
  \begin{minipage}{0.33\textwidth}
    \centering
    \includegraphics{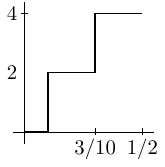}
  \end{minipage}%
  \caption{Levine-Tristram signatures of the $T(2, k)$ torus links for $k = 3,4,5$. The $x$-axis is $\theta$, and the $y$-axis is the signature
  function evaluated at $\omega = e^{2 \pi i \theta}$.}%
  \label{fig:signatures}
\end{figure}

\begin{lemma}\label{lem:signature-bounds}
  Let $\beta \in B_3^+$ be a positive braid
  of length $\ell(\beta)$
  and let $n$
  be a non-negative integer.
  Then
  \begin{enumerate}[\normalfont(i)]
    \item $d_\chi(\widehat \beta, T(2, 3)^n) \geq |2n - \ell(\beta)/3| - 2$,
    \item $d_\chi(\widehat \beta, T(2, 4)^n) \geq |3n - \ell(\beta)/2| - 2$,
    \item $d_\chi(\widehat \beta, T(2, 5)^n) \geq |4n - 3 \ell(\beta)/5| - 2$,
    \item $d_\chi(\widehat \beta, T(2, k)^n) \geq |(k-1)n - 2\ell(\beta)/3| - 4$
      for $k \geq 6$
      if $\widehat \beta$ is a torus link.
  \end{enumerate}
\end{lemma}

\begin{proof}
  Let $k \geq 3$.
  For $k \leq 5$, Lemma~\ref{lem:gambaudo-ghys} as well as the values
  $\sigma_k(T(2, k)) = k-1$ yield
  \[
    \max_{\omega \in S^1} |\sigma_\omega(T(2, k)^n) - \sigma_\omega(\widehat \beta)| \geq |\sigma_k(T(2, k)^n) - \sigma_k(\widehat \beta)|
    \geq |(k-1)n - 2 \theta_k \ell(\beta)| - 2,
  \]
  where $\theta_3 = 1/6$, $\theta_4 = 1/4$, $\theta_5 = 3/10$.
  For $k \geq 6$ on the other hand,
  we have
  \[
    \max_{\omega \in S^1} |\sigma_\omega(T(2, k)^n) - \sigma_\omega(\widehat \beta)|
    \geq |(k-1)n - 2\ell(\beta)/3| - 4.
  \]
  Here we use in addition that
  $|\sigma(T(3, m)) - 4m| \leq 4$, which follows from~\cite{gambaudo-ghys},
  specifically
  Proposition~5.2 and the fact that the homogeneous signature on $B_n$
  differs from the signature by at most $2n$, which is mentioned
  in their Section~5.1.
  Alternatively, explicit computations
  for the classical signatures of torus knots
  can be gathered from~\cite{glm},
  specifically Theorem~5.2.
  The desired results now follow from the well-known
  inequality
  \[
    d_\chi(L, L') \geq \max \left\{
    |\sigma_\omega(L) - \sigma_\omega(L') \mid \omega \in S^1 \right\}
  \]
  for all links $L, L'$.
\end{proof}

\section{Proofs of main results}\label{sec:proofs}
In this section, we combine the bounds from the previous two sections
to construct essentially minimal cobordisms between
the closure $\widehat \beta$ of a non-split positive braid
$\beta \in B_3^+$ or $T(3, m)$ and $T(2, k)^n$.
In the case $k \in \{3,4,5,6\}$, all these cobordisms factor through one specific link,
approximately
$T(2, k)^{[\ell(\beta)/k]}$ or $T(2, k)^{[2m/k]}$, respectively.

\begin{proof}[Proofs of Theorems~\ref{thm:3-braids},~\ref{thm:3-torus-knots}, and~\ref{thm:large-k}]
  We have
  $\chi(\widehat \beta) = -(\ell(\beta)-2)$ for all
  non-split positive
  braid words $\beta \in B_3^+$, and
  $\chi(T(2, k)^n) = -n(k-1)$ for all $k \geq 3$.
  Therefore,
  by the slice-Bennequin inequality~\cite{rudolph},
  we obtain
  \begin{equation}\label{eq:slice-bennequin}
    d_\chi(\widehat \beta, T(2, k)^n) \geq |\chi(T(2, k)^n) - \chi(\widehat \beta)|
    = |(k-1)n - \ell(\beta)| - 2.
  \end{equation}
  Combining this inequality with Lemma~\ref{lem:signature-bounds}
  yields a piecewise linear lower bound for
  $d_\chi(\widehat \beta, T(2, k)^n)$
  in $n$.
  A quick calculation shows
  \[
    d_\chi(\widehat \beta, T(2, k)^n) \geq \ell(\beta)/k + (k-1)|n - \ell(\beta)/k| - c_k
  \]
  for all $k \geq 3$, where $c_3 = c_4 = c_5 = 2$ and $c_6 = 4$.
  This takes care of the claimed lower bounds.


  Next, we turn to the upper bounds.
  Let $k \in \{3, 4, 5, 6\}$ and $\beta \in B_3^+$ be a non-split
  positive braid. If $k \geq 4$, we assume in addition that
  $\widehat \beta$ is a torus link $T(3, m)$,
  for which we have $\ell(\beta) = 2m$.
  Let $\widetilde C_3 = 10$, $\widetilde C_4 = 6$, $\widetilde C_5 = 19$, $\widetilde C_6 = 20$.
  Recall that we have shown
  in Section~\ref{sec:finding-subsurfaces}
  for all
  $N \geq 0 $ such that $kN \leq \ell(\beta) - \widetilde C_k$,
  that
  \begin{equation}\label{eq:small-N}
    d_\chi(\widehat \beta, T(2, k)^N) \leq \ell(\beta) - (k-1)N + 2.
  \end{equation}
  Let $N$ be the maximal integer satisfying this restriction,
  and let $n \geq N$ be arbitrary.
  We use the triangle inequality to derive the upper bound in question.
  It is easy to see that
  \(
    d_\chi(T(2, k)^N, T(2, k)^n) = (k-1)(n-N).
  \)
  We thus obtain
  \begin{align*}
    d_\chi(\widehat \beta, T(2, k)^n)
    & \leq d_\chi(\widehat \beta, T(2, k)^N) + d_\chi(T(2, k)^N, T(2, k)^n) \\
    & \leq \ell(\beta) - (k-1)N + 2 + (k-1)(n-N) \\
    &= \ell(\beta) + (k-1)n + 2 - 2(k-1)N.
  \end{align*}
  We have \(kN \geq \ell(\beta) - (\widetilde C_k + k-1)\)
  by maximality of $N$.
  This yields
  \begin{align*}
    d_\chi(\widehat \beta, T(2, k)^n)
    &\leq
   (k-1)n - \frac{(k-2)\ell(\beta)}{k} + 2(k-1)\frac{\widetilde C_k + k - 1}{k} + 2.
  \end{align*}
  Combining~(\ref{eq:small-N}) with this inequality
  proves Theorems~\ref{thm:3-braids} and~\ref{thm:3-torus-knots} for
  the constants $C_k = 2(k-1)(\widetilde C_k + k - 1)/k + 2$.
  We obtain the following explicit constants:
  \begin{enumerate}[(i)]
    \item $C_3 = 18$,
    \item $C_4 = 31/2$,
    \item $C_5 = 194/5$,
    \item $C_6 = 131/3$. \qedhere
  \end{enumerate}
\end{proof}

\begin{proof}[Proof of Theorem~\ref{thm:large-k}]
  The lower bounds come from
  Lemma~\ref{lem:signature-bounds} and Equation~(\ref{eq:slice-bennequin})
  (where we apply $m \geq n/3$).
  The upper bound for point (ii) was done in Section~\ref{sec:finding-subsurfaces},
  so only the upper bound for point (i) is missing.
  To this end,
  we describe a two-step process to construct
  a cobordism from $T(3,m)$ to $T(2, k)^n$.
  Let $n \geq m/3$.
  First, a cobordism of Euler characteristic
  $-(k-6)n$ transforms $T(2, k)^n$ into $T(2,6)^n$.
  Next, by Theorem~\ref{thm:3-torus-knots} we know that
  \[
    d_{\chi}(T(3, m), T(2, 6)^n) \leq 5n - 4m/3 + E_k(m, n).
  \]
  Putting the two together we obtain the claimed bound
  \begin{align*}
    d_\chi(T(3, m), T(2, k)^n)
    &\leq d_\chi(T(3,m), T(2, 6)^n) + d_\chi(T(2, 6)^{n}, T(2, k)^n) \\
    &\leq (k-1)n - 4m/3 + E_k(m, n),
  \end{align*}
  where $|E_k(m,n)| \leq |E_6(m,n)| \leq C_6$.
\end{proof}

\bibliographystyle{amsalpha}
\bibliography{biblio}

\end{document}